\documentclass{amsart}
\usepackage{amssymb, amsthm, graphicx}
\newtheorem{theorem}{Theorem}[section]
\newtheorem{lemma}[theorem]{Lemma}

\newtheorem{prop}[theorem]{Proposition}

\theoremstyle{definition}

\theoremstyle{remark}
\newtheorem{remark}[theorem]{Remark}

\numberwithin{equation}{section}

\newcommand{\Li}{{\operatorname{Li_2}}}
\newcommand{\li}{{\operatorname{li_2}}}
\newcommand{\Real}{{\operatorname{Re}}}
\newcommand{\Imag}{{\operatorname{Im}}}
\newcommand{\Vol}{{\operatorname{Vol}}}

\begin{document}
\title[Volume formulas for a spherical tetrahedron]{Volume formulas for a spherical  tetrahedron}
\author[Jun Murakami]{Jun Murakami}
\address{
Department of Mathematics\\ 
Faculty of Science and Engineering\\ 
Waseda University\\
3-4-1 Ohkubo, Shinjuku-ku\\ 
Tokyo 169-8555, JAPAN}
\email{murakami@waseda.jp}
\thanks{This research was partially supported by Grant-in-Aid for Scientific Research(C) 22540236 from JSPS. }
\begin{abstract}
The present paper gives two concrete formulas for the volume of an arbitrary spherical tetrahedron, which is in a 3-dimensional spherical space of constant curvature $+1$.  
One formula is given in terms of dihedral angles, and another one is given in terms of edge lengths.  
\end{abstract}
\maketitle
\section*{Introduction}
The calculation of the volume of an arbitrary tetrahedron in a 3-space of non-zero constant curvature is rather hard, and the first result is given by \cite{CH} in 1999 for hyperbolic tetrahedra.   
The papers \cite{MY} and \cite{MU} gave another formulas for hyperbolic tetrahedra, which are implicitly based on the quantum $6j$-symbol.   
Moreover, it was stated in \cite{MY} that an adequate analytic continuation of the obtained formula also applicable for a spherical tetrahedron.  
But, the formula is given by multi-valued functions, and it is not   described which stratum we should select for actual computation.  
On the other hand, volumes of spherical tetrahedra of special shapes are given by many people from old times,   
and the most recent work is \cite{KMP}, which gives a formula for a spherical tetrahedron having a small symmetry.  
\par
In the present paper, volume formulas for a spherical tetrahedron $T$  of general shape are given in Theorems 1.1 and 1.2.  
The formula in Theorem 1.1 is given in terms of dihedral angles, and the formula in Theorem 1.2  is given in terms of edge lengths.  
These formulas are obtained by improving those in \cite{MY}, \cite{MU},  and, by using the Schl\"afli differential equality, 
it is shown that the new formulas actually give the volume of  $T$ modulo $2\, \pi^2$.  
Please note that $2 \, \pi^2$ is the volume of $S^3$ with radius 1, which is the universal cover of any 3-dimensional spherical space of constant curvature $+1$.  
Since $T$ can be included in a 3-dimensional hemisphere, the volume of $T$ is less than $\pi^2$ and so we can compute the volume of $T$ actually from the formulas in Theorems 1.1 and 1.2.  
\section{Volume formulas}
\subsection{Volume formula in terms of dihedral angles}
Let $T$ be a spherical tetrahedron  and $\theta_1$, $\theta_2$, $\cdots$, $\theta_6$ be its dihedral angles at edges $e_1$, $e_2$, $\cdots$, $e_6$ respectively given in Figure 1.  
We assume that
$0 < \theta_j < \pi$ for $j = 1$, $2$, $\cdots$,  $6$.  
Let $a_1 = e^{i\theta_1}$, $a_2 = e^{i\theta_2}$, $\cdots$, $a_6 = e^{i\theta_6}$, and  
\begin{figure}[htb]
$$
\begin{matrix}
\quad e_1& & & \quad e_3 \\
& & \quad e_2 & \\[9pt]
& e_5 & & \\[6pt]
e_6 & & & e_4
\end{matrix}
\hskip-4.2cm
\raisebox{-1.3cm}{\includegraphics[scale=0.6]{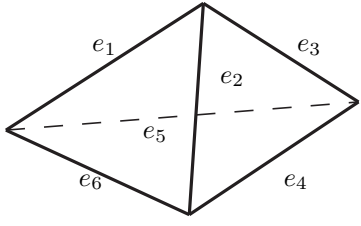}}
$$
\caption{Edges of $T$}
\label{figure:tetrahedron}
\end{figure}
\begin{multline*}
L(a_1, a_2, \cdots, a_6, z) 
=
\\
\dfrac{1}{2} \, \Big(
\Li(z) + \Li(a_1^{-1} \,a_2^{-1}\, a_4^{-1} \,a_5^{-1}\, z) + \Li( a_1^{-1} \,a_3^{-1} \,a_4^{-1} \,a_6^{-1} \, z) 
+ \Li(a_2^{-1} \,a_3^{-1} \,a_5^{-1} \,a_6^{-1} \,z) 
\\
- \Li(-a_1^{-1} \,a_2^{-1} \,a_3^{-1} \,z) - \Li(-a_1^{-1}\, a_5^{-1} \,a_6^{-1}\, z)
- \Li(-a_2^{-1} \,a_4^{-1} \,a_6^{-1} \,z) 
\\
- \Li(-a_3^{-1} \,a_4^{-1} \,a_5^{-1} \,z)
 +\sum_{j=1}^3\,
\log a_j \,\log a_{j+3}\Big),
\end{multline*}
where $\Li(z)$ is the dilogarithm function defined by analytic continuation of the following integral:
\begin{equation}
\Li(x) =
-\int_0^x \dfrac{\log(1-t)}{t} \, dt \qquad
\text{for a real number $x < 1$.}
\label{eq:dilog}
\end{equation}
The  analytic continuation of the right-hand side integral defines a multi-valued complex function $\li(z)$, and let $\Li(z)$ be the principal branch of $\li(z)$ which is the analytic continuation of \eqref{eq:dilog} on the region 
$\mathbb C \setminus \{x \in \mathbb R \mid x \geq 1 \}$.  
We also fix the principal branch of the log function as usual by the branch cut along the negative real axis.  
\par
We define the auxiliary parameter $z_0$ as follows:
\begin{equation}
z_0 = \dfrac{-q_1 + \sqrt{q_1^2 - 4 \,q_0\, q_2}}{2\, q_2},
\label{eq:z}
\end{equation}
where
\begin{equation*}
\begin{aligned}
q_0 &= 
 a_1 \, a_4 + a_2 \, a_5 + a_3 \, a_6
+ a_1 \, a_2 \, a_6 + a_1 \, a_3 \, a_5 + a_2 \, a_3 \, a_4 
\\
&\qquad\qquad\qquad\qquad\qquad\qquad \qquad
\qquad\qquad\quad
+ a_4 \, a_5 \, a_6 + a_1 \, a_2 \, a_3 \, a_4 \, a_5 \, a_6,
\\
q_1 &= -
(a_1 -a_1^{-1})(a_4 - a_4^{-1}) 
- (a_2 - a_2^{-1})(a_5 - a_5^{-1})
- (a_3 - a_3^{-1}) (a_6 - a_6^{-1}),
\\
q_2 &=
a_1^{-1}  a_4^{-1} + a_2^{-1} a_5^{-1} + a_3^{-1} a_6^{-1}
+ a_1^{-1}a_2 ^{-1} a_6^{-1} + a_1 ^{-1} a_3^{-1}a_5 ^{-1}+\\&\qquad\qquad\qquad\qquad\qquad 
a_2 ^{-1}a_3 ^{-1} a_4^{-1} + a_4^{-1} a_5^{-1} a_6 ^{-1}+
a_1^{-1} a_2 ^{-1} a_3 ^{-1} a_4 ^{-1}a_5 ^{-1} a_6^{-1}
.  
\end{aligned}
\label{eq:q}
\end{equation*}
Then $z_0$ is a solutions of 
\begin{equation} 
\exp\left(2 \, z \, \dfrac{\partial L}{\partial z}\right) = 1,
\label{eq:eqz}
\end{equation}
where $$
\exp\left(2 \, z \, \dfrac{\partial L}{\partial z}\right) 
=
\dfrac{(a_1 \, a_2 \, a_3+ z) \, (a_1\, a_5 \, a_6 +  z) \, 
(a_2\, a_4 \, a_6 + z) \, (a_3\, a_4 \, a_5 + z)}
{(1-z) \, (a_1 \, a_2 \, a_4\, a_5 - z) \, 
(a_1 \, a_3\, a_4 \, a_6 - z) \, (a_2\, a_3 \, a_5 \, a_6 - z)}
.
$$  
\par
Now we state the main result of this paper.  
\medskip
\par\noindent
\begin{theorem}  
Let $T$ be a spherical tetrahedron with dihedral angles $\theta_1$, $\theta_2$, $\cdots$, $\theta_6$ at the edges $e_1$, $e_2$, $\cdots$ $e_6$ given in Figure \ref{figure:tetrahedron}.  
Let $a_j = e^{i\theta_j}$ for $j = 1$, $2$, $\cdots$, $6$  and
let $\Vol(T)$ be the volume of $T$.  
Then 
\begin{multline*}
\Vol(T) = -\Real( L(a_1, a_2, \cdots, a_6, z_0) )
+
 \pi \left(\! \arg (-q_2)
 +
 \dfrac{1}{2} \, 
 \sum_{j=1}^6\, \theta_j\!\right)
 -\dfrac{3}{2}\, \pi^2
 \\
 \mod \ 2\, \pi^2 ,
\end{multline*}
where $\Real(z)$ is the real part of $z$ and $z_0$, $q_2$ given by \eqref{eq:z}.  
\end{theorem}
%
%
%
%
\subsection{Volume formula in terms of edge lengths}
Let $T$ be a  spherical tetrahedron with edge lengths $l_1$, $l_2$, $\cdots$, $l_6$ at the edges $e_1$, $e_2$, $\cdots$ $e_6$ respectively given  in Figure \ref{figure:tetrahedron}.   
Let $b_j = e^{i \, l_j}$ for $j=1$, $2$, $\cdots$, $6$
and $\widetilde L(b_1, b_2, b_3, b_4, b_5, b_6, z) 
\,=\,
L(-b_4^{-1}$, $-b_5^{-1}$, $ -b_6^{-1}$, $-b_1^{-1}$, $-b_2^{-1}$, $ -b_3^{-1}$, $z)$.  
Then the following formula holds.   
\medskip\par\noindent
\begin{theorem}  
For a spherical tetrahedron $T$ as above, 
\begin{multline*}
\Vol(T) = \Real\Big( \widetilde L(b_1, b_2, \cdots, b_6,  \widetilde z_0)\Big)
- \pi \arg (-\widetilde q_2)
\\
- \sum_{j=1}^6 l_j \, \left.\dfrac{\partial \,\Real\big(\widetilde L(b_1, b_2, \cdots, b_6,  z)\big)}{\partial l_j}\right|_{z = \widetilde z_0}
-
\dfrac{1}{2} \, \pi^2
\mod 2\, \pi^2,
\end{multline*}
where $\widetilde z_0$ and $\widetilde q_2$ are obtained from $z_0$ and $q_2$ in \eqref{eq:z} by substituting $-b_{j\pm3}^{-1}$ to $a_j$ for $j = 1$, $2$, $\cdots$, $6$.  
\end{theorem}
%
%
\section{Proof of the formulas}
\subsection{Gram matrices}
Let $T$ be a spherical tetrahedron with dihedral angles $\theta_1$, $\cdots$, $\theta_6$ as before.  
Let $G$ be the Gram matrix of $T$ defined by
$$
G = 
\begin{pmatrix}
1 & -\cos \theta_1 & -\cos \theta_2 & -\cos\theta_6 \\
-\cos\theta_1 & 1 & -\cos\theta_3 & -\cos\theta_5 \\
-\cos\theta_2 & -\cos\theta_3 & 1 & -\cos\theta_4 \\
-\cos\theta_6 & -\cos\theta_5 & -\cos\theta_4 & 1
\end{pmatrix}.  
$$
An actual computation shows that
the discriminant 
in \eqref{eq:z} is given by
\begin{equation}
q_1^2 - 4 \, q_0 \, q_2 
=
16 \, \det G,   
\label{eq:discriminant}
\end{equation}
which is positive since $T$ is spherical.
It is known\footnote{
The formula \eqref{equation:cosl} comes from the formula in p.8, l.4 of \cite{KMP} applied to the dual tetrahedron $T^*$.
It is a spherical version of the formula just below (5.1) in \cite{U}.  
}
that
\begin{equation}
\cos l_j = 
\dfrac{c_{pq}}{\sqrt{c_{pp} \, c_{qq}}}
\label{equation:cosl}
\end{equation}
and so we have
\begin{equation}
\exp (2\, i\,  l_j) 
=
\dfrac{2 \, c_{pq}^2- c_{pp}\, c_{qq}  +  2 \, i\, c_{pq}\, \sqrt{\det G} \, \sin\theta_j}{c_{pp}\, c_{qq}}
\label{eq:ls}
\end{equation}
by using the formula (5.1) in \cite{U} that is
$
c_{pq}^2 - c_{pp}\, c_{qq} = 
-\det G \, \sin^2 \theta_j  
$.
Here $p$ and $q$ denote the row and column of $G = (g_{ab})$ such that
$g_{p'q'} = -\cos \theta_j$, 
$\{p, q\} = \{1, 2, 3, 4\} \setminus \{p', q'\}$
and $c_{ab}$ is the cofactor of $G$, i.e. 
 $c_{ab} = (-1)^{a+b} \, \det G_{ab}$ where $G_{ab}$ is the submatrix obtained from $G$ by deleting its $a$-th row and $b$-th column.  
\subsection{Some functions and their properties}
Before proving the formulas, we introduce some functions and investigate their properties.  
Let $T$ be an abstract tetrahedron, $\theta_1$, $\theta_2$, $\cdots$, $\theta_6$ be its dihedral angles at the edges $e_1$, $e_2$, $\cdots$, $e_6$ as before, 
and 
\begin{multline*}
D_s = 
\{(\theta_1, \theta_2, \cdots, \theta_6) \in
(0, \pi)^6 \subset \mathbb R^6 \mid
\text{$\theta_1$, $\theta_2$, $\cdots$, $\theta_6$} 
\\
\text{correspond to the dihedral angles of a spherical tetrahedron}\}.   
\end{multline*}
Let $a_j=e^{i\theta_j}$ for $j=1, 2, \cdots, 6$, 
\begin{multline*}
\Delta_0(x,y, z) 
=
\\
-\dfrac{1}{4} \Big(\Li(-x y^{-1} z^{-1}) + \Li(-x^{-1} y z^{-1})
+\Li(-x^{-1} y^{-1} z)+ \Li(-xy  z)\Big),
\end{multline*}
\begin{multline*}
\Delta(a_1, a_2, \cdots, a_6) = 
\Delta_0(a_1, a_2, a_3) + \Delta_0(a_1, a_5, a_6) 
+ 
\Delta_0(a_2, a_4, a_6)
\\
 +\Delta_0(a_3, a_4, a_5)
-\dfrac{1}{2} \,
\sum_{j=1}^6 \big(\log a_j)^2 
,
\end{multline*}
$$
U(a_1, a_2, \cdots, a_6, z) =
 L(a_1, a_2, \cdots, a_6, z) + \Delta(a_1, a_2, \cdots, a_6)
,  
$$
and
\begin{multline*}
V(a_1, a_2, a_3, a_4, a_5, a_6) 
= 
\\
-U(a_1, a_2, a_3, a_4, a_5, a_6, z_0) 
+ 
\pi\, i \, \left(\log z_0 - \sum_{j=1}^6 \, \log a_j\right)
-
\dfrac{13}{6}\, \pi^2
.
\end{multline*}
\par
\begin{lemma}
The function $\Delta(a_1, a_2, \cdots, a_6)$ is analytic on $D_s$ and 
the imaginary part of $4 a_j  \frac{\partial \Delta}{\partial a_j}$ is given by
$$
\Imag\left(4 \, a_j\, \dfrac{\partial \Delta}{\partial a_j}\right)
=
-2 \, \pi.  
$$
\label{lemma:imagdelta}
\end{lemma}
\begin{proof}
We show for the case $j=1$.  
For the function $\Delta$, 
\begin{equation*}
a_1 \, \dfrac{\partial \Delta}{\partial a_1}
=
a_1 \, \dfrac{\partial \Delta_0(a_1, a_2, a_3)}{\partial a_1}
+
a_1 \, \dfrac{\partial \Delta_0(a_1, a_5, a_6)}{\partial a_1}
-
\log a_1
\end{equation*}
and
\begin{multline*}
a_1 \, \dfrac{\partial \Delta_0(a_1, a_p, a_q)}{\partial a_1}
=
\\
\dfrac{1}{4} \, 
\left(\!\log(1 + \dfrac{a_1}{a_p  a_q}) 
- 
\log(1+\dfrac{a_p}{a_1  a_q})
- 
\log(1 + \dfrac{a_q}{a_1  a_p}) 
+ 
\log(1 + {a_1  a_p  a_q})\!\right)  
\end{multline*}
for $\{p,q\} = \{2,3\}$, $\{5,6\}$. 
The imaginary part $\Imag \log(1+e^{ i \theta})$ 
is given by
$$
\Imag \log(1 + e^{i\theta}) = \begin{cases}
\dfrac{\theta}{2} & \text{ if $-\pi < \theta < \pi$}, \\[12pt]
\dfrac{\theta}{2}-\pi & \text{ if $\pi < \theta < 3\, \pi$}.
\end{cases} 
$$
Let $\theta_u$, $\theta_v$, $\theta_w$ be three dihedral angles at three edges having a vertex in common.  
Then they satisfy
\begin{equation}
0 < \theta_u + \theta_v - \theta_w,\ \theta_u - \theta_v + \theta_w,\  -\theta_u + \theta_v + \theta_w
< \pi, 
\quad
\pi < \theta_u + \theta_v + \theta_w < 3\,\pi.  
\label{equation:admissible}
\end{equation}  
Hence $\Delta_0(a_1, a_p, a_q)$ is analytic on $D_s$ and we have
$$
\Imag \left(a_1 \, \dfrac{\partial \Delta_0(a_1, a_p, a_q)}{\partial a_1}\right)
=
\dfrac{\theta_1}{2} - \dfrac{\pi}{4},
\qquad
\Imag\left(4 \, a_1\, \frac{\partial \Delta}{\partial a_1}\right)
=
-2 \, \pi.
$$
Moreover,  $\Delta$ is analytic on $D_s$
because none of the imaginary parts of the log terms of $\Delta$ attains neither $\pi$ nor $-\pi$ on $D_s$.  
\end{proof}
\begin{lemma}
The function $L(a_1, a_2, \cdots, a_6, z_0(a_1, a_2, \cdots, a_6))$ 
 is analytic on $D_s$, and so 
 $U(a_1, a_2, \cdots, a_6, z_0(a_1, a_2, \cdots, a_6))$ 
 is analytic on $D_s$.  
\label{lemma:analytic}
\end{lemma}
\begin{proof}
We know that
$
|z_0| < 1 
$
because,
for   $q_0$, $q_1$, $q_2$ in \eqref{eq:z}, 
$q_1$ is a real number and 
 $q_0 q_2 = q_0 \overline{q_0}=|q_0|^2$ is a positive real number, and $q_1^2 - 4 q_0q_2$ is also a positive real number by \eqref{eq:discriminant}.
This implies that, for $w \in \mathbb C$ with $|w|=1$,   
$|w\, z_0| < 1$.  
Noting that $\Li(z)$ is analytic on the unit open disk $\{z \in {\mathbb C} \mid |z| < 1\}$, all the dilog terms of $L$  are analytic on $D_s$
since $|a_1| = \cdots = |a_6| = 1$.  
\end{proof}
\begin{lemma}
The differential $\frac{\partial U}{\partial z}$ satisfies   
$
\left.z_0 \frac{\partial U}{\partial z}\right|_{z= z_0}
=
\pi \, i
$.
\label{lemma:eqz}
\end{lemma}
\begin{proof}
Since $\frac{\partial U}{\partial z} = \frac{\partial L}{\partial z}$ and $z_0$ is a solution of the equation \eqref{eq:eqz}, 
$
\left.z_0 \, \frac{\partial U}{\partial z}\right|_{z= z_0}$
$=
k\, \pi \, i
$
for some integer constant $k$ because $U$ is analytic on $D_s$ by the above lemma.  
Let $T_{\frac{\pi}{2}}$ be the regular spherical tetrahedron with edge lengths $\pi/2$.  Then $\theta_j = \pi/2$,  $a_j = i$ for $j=1, \cdots, 6$, $z_0 = (i+1)/2$ and
$$
\left.z_0 \, \dfrac{\partial U}{\partial z}\right|_{z= z_0}
=
\dfrac{1}{2} \, \left(-4\,\log \dfrac{1-i}{2} + 4\,\log \dfrac{1+i}{2}\right)
=
\pi \, i.  
$$
Hence $\left.z_0  \frac{\partial U}{\partial z}\right|_{z= z_0} = \pi i$ 
for all the spherical tetrahedron.  
\end{proof}
Now, we show the following proposition for $V$ corresponding to the Schl\"afli differential equality 
\begin{equation}
d\, \Vol(T) 
= 
\sum_{j=1}^6
\dfrac{l_j}{2} \,  d  \theta_j,  
\end{equation}
which is a fundamental tool to analize the volume.  
For example, see \cite{Mi}. 
\begin{prop}
The function $V$ satisfies
$
\frac{\partial V}{\partial \theta_j}
= 
{l_j}/{2}$ for $j=1, 2, \cdots, 6$.  
\label{prop:lengths}
\end{prop}
\begin{proof}
Let 
$
\varphi = 
\exp\Big(4 \, a_1 \, \frac{\partial \Delta}{\partial a_1}\Big)
$
and
$
\psi =
 \exp\Big(\left.2 \, a_1 \, \frac{\partial L}{\partial a_1}\right|_{z = z_0}\Big)
 $,
 then 
 $$
 \varphi = \frac{(a_1+a_2\, a_3) (a_1\,  a_2\,  a_3+1)
   (a_1+a_5 \, a_6) (a_1 \, a_5 \, a_6 +1)}{(a_1 \, a_2 +a_3) (a_1 \,  a_3 +a_2) (a_1 \,  a_5 +a_6)
   (a_1\, a_6+a_5)}, 
   $$
   $$
 \psi = 
   \frac{(a_1 \, a_2 \, a_4 \, a_5  -z_0) (a_1 \, a_3 \, a_4 \, a_6 - z_0)}{a_4\,  (a_1 \, a_2 \, a_3   +z_0) (a_1\,  a_5 \, a_6  
   + z_0)}.    
  $$
An actual computation and \eqref{eq:ls} show that 
\begin{multline*}
\left.\exp\left(\!4  a_1  \dfrac{\partial U}{\partial a_1}\!\right)
\right|_{z = z_0}
\!\!\!\!
  =
\varphi \psi^2
=
\dfrac{2 \, c_{34}^2 -  c_{33} c_{44} + 2 i   c_{34}  \sqrt{\det G} \, \sin \theta_1}{c_{33}\, c_{44}} 
=
\exp(2 \, l_1 \, i).  
\end{multline*}
Hence we get
$
\left.a_1  \frac{\partial U}{\partial a_1}\right|_{z=z_0}
=
i \, (l_1 + k\,  \pi)/2
$
for some integer constant $k$  
because $U$ is analytic on
$D_s$ by Lemma \ref{lemma:analytic}. 
For the tetrahedron $T_{\frac{\pi}{2}}$ given in the proof of Lemma \ref{lemma:eqz}, 
$
l_1= \pi/2
$ 
and 
$
\left.a_1  \frac{\partial U}{\partial a_1}\right|_{z=z_0} 
= 
-3\,\pi\, i/4
$, 
which means that 
$k=-2$   
and
$
\left.a_1 \, \frac{\partial U}{\partial a_1}\right|_{z=z_0} 
= 
i\, (l_1 - 2\, \pi)/2.  
$
According to 
$\frac{\partial U}{\partial \theta_1}
=
i  a_1  \frac{\partial U}{\partial a_1}$, 
we have
\begin{equation}
\left.\dfrac{\partial U}{\partial \theta_1}\right|_{z=z_0} 
= 
\dfrac{1}{2}\, (2\, \pi-l_1).
\label{eq:length}
\end{equation}
Therefore
\begin{multline*}
\dfrac{\partial}{\partial \theta_1} \left(
-U\big(a_1, \cdots, a_6, z_0(a_1, \cdots, a_6)\big) 
+ \pi \, i \,  \Big( \log z_0 - \sum_{j=1}^6\, \log a_j\Big)
\right)
=
\\
\dfrac{l_1}{2}-
\left.
\dfrac{\partial z_0}{\partial \theta_1}\, \dfrac{\partial U}{\partial z}
\right|_{z=z_0}
+
\pi \, i\, \dfrac{\partial z_0}{\partial \theta_1} 
\,
\dfrac{1}{z_0 }.
 \end{multline*}
Since 
$ \left.
 \frac{\partial U}{\partial z}\right|_{z=z_0} =
{i \, \pi}/{z_0}$
 by Lemma \ref{lemma:eqz}, 
 we get
 $
\frac{\partial V}{\partial \theta_1}
=
 l_1/2
$.  
\end{proof}
\subsection{Proof of the formula in terms of dihedral angles}
We first give a formula by complex analytic functions.  
\begin{prop}
Let $T$ be a spherical tetrahedron with dihedral angles $\theta_1$, $\theta_2$, $\cdots$, $\theta_6$ at the edges $e_1$, $e_2$, $\cdots$ $e_6$ as in Figure \ref{figure:tetrahedron}.  
Let $a_j = e^{i\theta_j}$ for $j = 1$, $2$, $\cdots$, $6$ as before and
let $\Vol(T)$ be the volume of $T$.  
Then 
$$
\Vol(T) = V(a_1, a_2, a_3, a_4, a_5,  a_6) \ \mod \ 2\, \pi^2.
$$  
\label{prop:edge}
\end{prop}
\begin{proof}
For the tetrahedron $T_{\frac{\pi}{2}}$ in the proof of Lemma \ref{lemma:eqz}, we have $a_j= i$, $z_0 = \frac{1 + i}{2}$ and  $V(i, i, i, i, i, i, i) = \pi^2/8 = \Vol(T_{\frac{\pi}{2}})$ since $T_{\frac{\pi}{2}}$ is one-sixteenth of $S^3$ and the volume of $S^3$ with radius $1$ is  $2 \pi^2$.  
Because $V$ is analytic on some neighborhood $N$ of $T_{\frac{\pi}{2}}$ in $D_s$, two functions $V$ and $\Vol$ are identical on $N$ by Proposition \ref{prop:lengths} and Sch\"afli differential equality.  
Moreover, $\Vol$ is analytic on $D_s$ and so it is given by an adequate analytic continuation of $V$.  
We already showed in previous lemmas that all the terms in $V$ except \  $\pi  i  \log z_0$ \ are analytic on $D_s$, and the analytic continuation of \ 
$\pi i  \log z_0$ is  \ 
$
\pi  i  \log z_0+ 2 k \pi^2
$ \ 
for some integer $k$.  
Hence we get the proposition.  
\end{proof}
\par\noindent
{\it Proof of Theorem 1.1}.  
We prove Theorem 1.1 by investigating the real part of $V$.  
For $\theta \in [0, 2\, \pi] \subset \mathbb R$,
the real part of $\Li(e^{i\theta})$ is given by
$
\Real\big(\Li(e^{i\theta})\big) 
=
\Real\big(\Li(e^{-i\theta})\big) 
=
\theta^2/4 - \pi \, \theta/2 +
{\pi^2}/{6}
$.  
Substituting this to each dilog function of $\Real(\Delta(a_1$, $a_2$, $\cdots$, $a_6))$, we get
$
\Real(\Delta(a_1, a_2, \cdots, a_6))
=
 - 2 \, \pi^2 /3 
+
\sum_{j=1}^6 \pi \theta_j/2
$
by using \eqref{equation:admissible}.
We also know that $\Imag \log z_0 = -\arg (-q_2)$ since the numerator of $z_0$ in \eqref{eq:z} is a negative real number.  
Hence we get Theorem 1.1 from Proposition \ref{prop:edge}.  
\qed
\begin{remark}
The function $V$ is non-continuous at the points where the values of $q_2$ are positive real numbers.  
\end{remark}
\subsection{Proof of the formula in terms of edge lengths}
We use the notations in Subsection 2.2.  
\par\noindent{\it Proof of Theorem 1.2.}
Let $\theta_1$, $\theta_2$, $\cdots$, $\theta_6$ be the dihedral angles at the edges $e_1$, $e_2$, $\cdots$, $e_6$ of $T$ and
let  $T^*$ be the dual tetrahedron of $T$ given by
\cite[p.294]{Mi}.  
Then the  dihedral angles of $T^*$ are $\pi- l_4$, $\pi - l_5$, $\pi - l_6$, $\pi - l_1$, $\pi - l_2$, $\pi -l_3$ and  
the edge length of $T^*$ are $\pi- \theta_4$, $\pi - \theta_5$, $\pi - \theta_6$, $\pi - \theta_1$, $\pi -\theta_2$, $\pi -\theta_3$.
The relation of volumes of $T$ and $T^*$ is given by  \cite[p.294]{Mi} as follows:  
$$
\Vol(T) + \Vol(T^*) + \dfrac{1}{2} \,
\sum_{j=1}^6 \, l_j \, (\pi-\theta_j)
=
\pi^2.      
$$
By Theorem 1.1, we have
\begin{multline*}
\Vol(T^*) 
=
-
\Real(\widetilde L(b_1, b_2, \cdots, b_6, \widetilde{z}_0))
+ \pi\,\left(\arg (-\widetilde q_2) + \dfrac{1}{2}\, \sum_{j=1}^6 (\pi - l_j)\right)
-
\dfrac{3}{2}\, \pi^2
\\
\mod 2\, \pi^2.
\end{multline*}
Because 
$\left.\frac{\partial }{\partial (\pi - l_j)}U(-b_4^{-1}, -b_5^{-1}, -b_6^{-1}, -b_1^{-1}, -b_2^{-1}, -b_3^{-1}, z) \right|_{z = \widetilde z_0}
=
\big( 2\, \pi-(\pi - \theta_j)\big)/2$ 
by \eqref{eq:length}
 and
 $
 \frac{\partial}{\partial (\pi - l_j)}\Real\left(\Delta(-b_4^{-1}, -b_5^{-1}, -b_6^{-1}, -b_1^{-1}, -b_2^{-1}, -b_3^{-1})\right) = \pi/2$ by Lemma \ref{lemma:imagdelta}, 
 we know that
 $
 \left.\frac{\partial \, \Real\widetilde L}{\partial l_j}\right|_{z = \widetilde z_0}
 =
-\theta_j/2.  
 $
Hence
\begin{multline*}
\Vol(T)
=
\Real\big(\widetilde L(b_1, b_2, \cdots, b_6, \widetilde{z}_0)\big)
- \pi  \, \arg (-\widetilde q_2)
\\
- \sum_{j=1}^6\,  l_j \, \left.\dfrac{\partial \,\Real\big(\widetilde L(b_1, b_2, \cdots, b_6,{z})\big)}{\partial l_j}\right|_{z = \widetilde z_0}
-
\dfrac{1}{2} \, \pi^2
\mod 2\, \pi^2, 
\end{multline*}
and we get Theorem 1.2.  
\qed

\end{document}